\documentclass[12pt]{amsart}

\usepackage{amsmath,amssymb,enumitem,verbatim,stmaryrd,xcolor,microtype,graphicx,mathscinet,mathtools,dirtytalk}
\usepackage[T1]{fontenc}
\usepackage[utf8]{inputenc}
\usepackage[english]{babel}
\usepackage[top=3.2cm,bottom=3.2cm,left=3.2cm,right=3.2cm]{geometry}
\usepackage[bookmarksdepth=2,linktoc=page,colorlinks,linkcolor={red!80!black},citecolor={red!80!black},urlcolor={blue!80!black}]{hyperref}
\usepackage{tikz}\usetikzlibrary{matrix,arrows,decorations.markings}
\usepackage{tikz-cd}
\usepackage[all]{xy}
\usepackage[mathcal]{euscript}                               
\usepackage{mathptmx}                                        
\usepackage[backgroundcolor=yellow,linecolor=yellow,textsize=footnotesize]{todonotes} \makeatletter \providecommand \@dotsep{5} \def\listtodoname{List of Todos} \def\listoftodos{\@starttoc{tdo}\listtodoname} 
 \makeatother 

\usepackage{etoolbox}
\makeatletter
\patchcmd{\@startsection}{\@afterindenttrue}{\@afterindentfalse}{}{}             
\patchcmd{\part}{\bfseries}{\bfseries\LARGE}{}{}
\patchcmd{\section}{\scshape}{\bfseries}{}{}\renewcommand{\@secnumfont}{\bfseries} 
\patchcmd{\@settitle}{\uppercasenonmath\@title}{\large}{}{}
\patchcmd{\@setauthors}{\MakeUppercase}{}{}{}
\addto{\captionsenglish}{} 
\addto{\captionsenglish}{} 
\addto{\captionsenglish}{} 
\makeatother

\usepackage{fancyhdr}

\theoremstyle{plain}
\newtheorem{thm}{Theorem}[section]

\newtheorem{lemma}[thm]{Lemma}
\newtheorem{prop}[thm]{Proposition}
\newtheorem{thmA}{Theorem}  
\newtheorem{propA}[thmA]{Proposition}  
\newtheorem{lemA}[thmA]{Lemma}  

\newtheorem*{thm*}{Theorem}

\theoremstyle{definition}
\newtheorem{df}[thm]{Definition}
\newtheorem{rem}[thm]{Remark}
\newtheorem{ex}[thm]{Example}

\theoremstyle{remark}

\DeclareRobustCommand{\gobblefour}[5]{}    

\DeclareFontFamily{OT1}{pzc}{}                                
\DeclareFontShape{OT1}{pzc}{m}{it}{<-> s * [1.10] pzcmi7t}{}
\DeclareMathAlphabet{\mathpzc}{OT1}{pzc}{m}{it}
\DeclareSymbolFont{sfoperators}{OT1}{bch}{m}{n} \DeclareSymbolFontAlphabet{\mathsf}{sfoperators} \makeatletter\def\operator@font{\mathgroup\symsfoperators}\makeatother 
\DeclareSymbolFont{cmletters}{OML}{cmm}{m}{it}              
\DeclareSymbolFont{cmsymbols}{OMS}{cmsy}{m}{n}
\DeclareSymbolFont{cmlargesymbols}{OMX}{cmex}{m}{n}
\DeclareMathSymbol{\myjmath}{\mathord}{cmletters}{"7C}     \let\jmath\myjmath 
\DeclareMathSymbol{\myamalg}{\mathbin}{cmsymbols}{"71}     
\DeclareMathSymbol{\mycoprod}{\mathop}{cmlargesymbols}{"60}
\DeclareMathSymbol{\myalpha}{\mathord}{cmletters}{"0B}     \let\alpha\myalpha 
\DeclareMathSymbol{\mybeta}{\mathord}{cmletters}{"0C}      \let\beta\mybeta
\DeclareMathSymbol{\mygamma}{\mathord}{cmletters}{"0D}     \let\gamma\mygamma
\DeclareMathSymbol{\mydelta}{\mathord}{cmletters}{"0E}     \let\delta\mydelta
\DeclareMathSymbol{\myepsilon}{\mathord}{cmletters}{"0F}   \let\epsilon\myepsilon
\DeclareMathSymbol{\myzeta}{\mathord}{cmletters}{"10}      \let\zeta\myzeta
\DeclareMathSymbol{\myeta}{\mathord}{cmletters}{"11}       \let\eta\myeta
\DeclareMathSymbol{\mytheta}{\mathord}{cmletters}{"12}     \let\theta\mytheta
\DeclareMathSymbol{\myiota}{\mathord}{cmletters}{"13}      \let\iota\myiota
\DeclareMathSymbol{\mykappa}{\mathord}{cmletters}{"14}     \let\kappa\mykappa
\DeclareMathSymbol{\mylambda}{\mathord}{cmletters}{"15}    \let\lambda\mylambda
\DeclareMathSymbol{\mymu}{\mathord}{cmletters}{"16}        \let\mu\mymu
\DeclareMathSymbol{\mynu}{\mathord}{cmletters}{"17}        \let\nu\mynu
\DeclareMathSymbol{\myxi}{\mathord}{cmletters}{"18}        \let\xi\myxi
\DeclareMathSymbol{\mypi}{\mathord}{cmletters}{"19}        \let\pi\mypi
\DeclareMathSymbol{\myrho}{\mathord}{cmletters}{"1A}       \let\rho\myrho
\DeclareMathSymbol{\mysigma}{\mathord}{cmletters}{"1B}     \let\sigma\mysigma
\DeclareMathSymbol{\mytau}{\mathord}{cmletters}{"1C}       \let\tau\mytau
\DeclareMathSymbol{\myupsilon}{\mathord}{cmletters}{"1D}   \let\upsilon\myupsilon
\DeclareMathSymbol{\myphi}{\mathord}{cmletters}{"1E}       \let\phi\myphi
\DeclareMathSymbol{\mychi}{\mathord}{cmletters}{"1F}       \let\chi\mychi
\DeclareMathSymbol{\mypsi}{\mathord}{cmletters}{"20}       \let\psi\mypsi
\DeclareMathSymbol{\myomega}{\mathord}{cmletters}{"21}     \let\omega\myomega
\DeclareMathSymbol{\myvarepsilon}{\mathord}{cmletters}{"22}\let\varepsilon\myvarepsilon
\DeclareMathSymbol{\myvartheta}{\mathord}{cmletters}{"23}  \let\vartheta\myvartheta
\DeclareMathSymbol{\myvarpi}{\mathord}{cmletters}{"24}     \let\varpi\myvarpi
\DeclareMathSymbol{\myvarrho}{\mathord}{cmletters}{"25}    \let\varrho\myvarrho
\DeclareMathSymbol{\myvarsigma}{\mathord}{cmletters}{"26}  \let\varsigma\myvarsigma
\DeclareMathSymbol{\myvarphi}{\mathord}{cmletters}{"27}    \let\varphi\myvarphi

\DeclareMathOperator{\Pol}{Poly}
\DeclareMathOperator{\mult}{mult}
\DeclareMathOperator*{\hypersum}{\,\raisebox{-2.2pt}{\larger[2]{$\boxplus$}}\,}
\DeclareMathOperator*{\hyperprod}{\,\raisebox{-1.4pt}{\scalebox{1.5}{$\boxdot$}}\,}
\DeclareMathOperator*{\hypercross}{\,\raisebox{-1.4pt}{\scalebox{1.25}{$\boxtimes$}}\,}

\newcommand\C{{\mathbb C}}
\newcommand\F{{\mathbb F}}

\newcommand\K{{\mathbb K}}
\newcommand\N{{\mathbb N}}
\renewcommand\P{{\mathbb P}}
\newcommand\Q{{\mathbb Q}}
\newcommand\R{{\mathbb R}}
\renewcommand\S{{\mathbb S}}
\newcommand\T{{\mathbb T}}

\newcommand\W{{\mathbb W}}
\newcommand\Z{{\mathbb Z}}

\newcommand\cP{{\mathcal P}}

\newcommand\ord{\textup{ord}}

\newcommand\sign{\textup{sign}}

\renewcommand\geq{\geqslant}
\renewcommand\leq{\leqslant}

\newcommand{\hyperplus}{\mathrel{\,\raisebox{-1.1pt}{\larger[-0]{$\boxplus$}}\,}}
\newcommand{\hyperdot}{\,\raisebox{-0.3pt}{\larger[+1]{$\boxdot$}}\,}
\newcommand{\hypertimes}{\,\raisebox{-1.1pt}{\larger[+0]{$\boxtimes$}}\,}
\newdir{ >}{{}*!/-5pt/@^{(}}\newcommand{\arincl}[1]{\ar@{ >->}@<-0,0ex>#1} 

\title{Descartes' rule of signs, Newton polygons, and polynomials over hyperfields}
 
\author{Matthew Baker}
\address{\rm Matthew Baker, School of Mathematics, Georgia Institute of Technology, Atlanta, USA}
\email{mbaker@math.gatech.edu}

\author{Oliver Lorscheid}
\address{\rm Oliver Lorscheid, Instituto Nacional de Matem\'atica Pura e Aplicada, Rio de Janeiro, Brazil}
\email{oliver@impa.br}

\begin{document}

\begin{abstract} 
 In this note, we develop a theory of multiplicities of roots for polynomials over hyperfields and use this to provide a unified and conceptual proof of both Descartes' rule of signs and
 Newton's ``polygon rule''.
\end{abstract}

\thanks{{\bf Acknowledgements.} We thank Philipp Jell for pointing out Remark~\ref{rem:infiniteroots}, Ziqi Liu for sharing with us the example from Appendix \ref{subsection: polynomial hyperings are not associative}, and St\'ephane Gaubert and an anonymous referee for several remarks, including explanations on the relations of certain hyperfields with supertropical and symmetrized semirings. We also thank Trevor Gunn, Jaiung Jun, Yoav Len, Sam Payne and Thor Wittich for their comments on an early draft version of this manuscript. The first author was supported by NSF Grant DMS-1529573 and a Simons Foundation Fellowship.}

\maketitle


\section*{Introduction}
\label{intro}



Given a real polynomial $p \in \R[T]$, {\em Descartes' rule of signs} provides an upper bound for the number of positive (resp. negative) real roots of $p$ in terms of the signs of the coefficients of $p$.  Specifically, the number of positive real roots of $p$ (counting multiplicities) is bounded above by the number of {\em sign changes} in the coefficients of $p(T)$, and the number of negative roots is bounded above by the number of sign changes in the coefficients of $p(-T)$.  

\medskip

Another classical ``rule'', which is less well known to mathematicians in general but is used quite often in number theory, is {\em Newton's polygon rule}.  This rule concerns polynomials over fields equipped with a {\em valuation}, which is a function $v : K \to \R \cup \{ \infty \}$ satisfying
\begin{itemize}
\item $v(a)=\infty$ if and only if $a=0$
\item $v(ab)=v(a)+v(b)$
\item $v(a+b) \geq \min\{ v(a),v(b) \}$, with equality if $v(a) \neq v(b)$
\end{itemize}
for all $a,b \in K$.

An example is the $p$-adic valuation $v_p$ on $\Q$, where $p$ is a prime number, given by the formula $v_p(s/t) = \ord_p(s) - \ord_p(t)$, where $\ord_p(n)$ is the maximum power of $p$ dividing a nonzero integer $n$.

Another example is the $T$-adic valuation $v_T$ on $k(T)$, for any field $k$, given by $v_T(f/g) = \ord_T(f) - \ord_T(g)$, where $\ord_T(f)$ is the maximum power of $T$ dividing a nonzero polynomial $f \in k[T]$.

\medskip

Given a field $K$, a valuation $v$ on $K$, and a polynomial $p \in K[T]$, Newton's polygon rule provides an upper bound for the number of roots (again counting multiplicities) of $p$ having a given valuation $s$ in terms of the valuations of the coefficients of $p$.  In this case, the rule is more complicated than in the case $K=\R$; the upper bound $\nu_s(p)$ is the length of the projection to the $x$-axis of the unique segment of the {\em Newton polygon} of $p$ having slope $-s$ (if such a segment exists), or zero (if no such segment exists).  (See Definition~\ref{def:NP} for a definition of the Newton polygon of $p$.) If $p$ splits into linear factors over $K$, the upper bound provided by Newton's rule is in fact an equality.

\medskip

The purpose of this note is to provide a conceptual unification of these two similar-looking, and yet seemingly rather different, upper bounds via the theory of {\em hyperfields}.

\medskip

Hyperfields are a generalization of fields where addition is allowed to be multi-valued.
Given a hyperfield $F$ and a polynomial $p$ over $F$ (by which we simply mean a formal expression of the form $\sum_{i=0}^n c_i T^i$ with $c_i \in F$), we will define what it means for an element $a \in F$ to be a {\em root of $p$}, and more generally we will define the {\em multiplicity} of $a$ as a root of $p$.  We denote this multiplicity by  $\mult_a(p)$.

\medskip

In the case of the {\em sign hyperfield} $\S$, we will find that the multiplicity of $1$ as a root of $p \in \S[T]$ is just the number of sign changes in the coefficients of $p$.
And in the case of the {\em tropical hyperfield} $\T$, we will see that the multiplicity of $s$ as a root of $p \in \T[T]$ is precisely $\nu_s(p)$.

\medskip

Moreover, if $K$ is a field (considered as a hyperfield), $f:K\to F$ is a hyperfield homomorphism, and $p \in K[T]$ is a polynomial, our definition of multiplicities will imply that the multiplicity of $a \in F$ as a root of $f(p)$ is at least the sum of the multiplicities $\mult_b(p)$ over all preimages $b \in f^{-1}(a)$.
Applying this fact to the natural homomorphism ${\rm sign} : \R \to \S$ will yield Descartes' rule of signs, and given a valuation $v$ on a field $K$ (which is the same thing as a homomorphism from $K$ to $\T$) we will recover Newton's polygon rule.

\medskip

\subsection*{Content overview}
In section \ref{section: statement of the main results}, we explain the overall idea behind our simultaneous proof of Descartes' rule of signs and Newton's polygon rule. In section~\ref{section: hyperfields}, we give a rigorous definition of hyperfields and a proof of Lemma~\ref{lem: roots} and Proposition \ref{propA}.  The above-mentioned interpretation of the multiplicities of roots over the sign hyperfield is established in section \ref{section: multiplicities over the sign hyperfield}, and for the tropical hyperfield this is worked out in section \ref{section: multiplicities of tropical roots}.

In Appendix \ref{section: polynomial algebras over hyperfields}, we investigate different possible notions of ``polynomial algebra'' over a hyperfield $F$.  We argue that while the older theory of ``additive-multiplicative hyperrings'' leads to a rather badly behaved notion, the second author's theory of ordered blueprints furnishes an efficient and satisfying (at least from a categorical perspective) theory of polynomial algebras over hyperfields.  We also discuss how the theory described in the body of this paper generalizes neatly to ordered blue fields which satisfy a ``reversibility'' axiom.


\section{Statement of the main results}
\label{section: statement of the main results}

\subsection*{Introduction to hyperfields}

We already mentioned that hyperfields\footnote{Marc Krasner introduced hyperrings and hyperfields in \cite{Krasner56}, and since then many authors have considered more general versions of these notions. The hyperfields considered in this paper, however, will all be hyperfields in the original sense of Krasner.} are a generalization of fields where addition is allowed to be multi-valued.
Somewhat more precisely, in a hyperfield $F$ addition is replaced by a {\em hyperoperation} $\hyperplus$, which is a map
\[
 \hyperplus: \quad F\times F \quad \longrightarrow \quad \cP(F)
\]
into the power set $\cP(F)$ of $F$.  The multiplication and hyperaddition operations on $F$ are required to satisfy various axioms, the most non-obvious of which is that there should be a distinguished neutral element $0 \in F$ such that for each $x \in F$, there is a unique $-x \in F$ such that $0 \in x \hyperplus (-x)$. We will give a more precise definition in section~\ref{section: hyperfields}; for now we content ourselves with some examples.

The three most important examples of hyperfields, for the purposes of this paper, are the following:

\begin{itemize}
 \item Every field $K$ is tautologically a hyperfield by defining $a\hyperplus b=\{a+b\}$. 
 \item The \emph{sign hyperfield} $\S$ consists of three elements $\{0,1,-1\}$ with the usual multiplication and hyperaddition characterized by the rules $1\hyperplus 1=\{1\}$, $-1\hyperplus -1=\{-1\}$ and $1\hyperplus-1=\{0,1,-1\}$. 
 \item The \emph{tropical hyperfield}
 $\T$ has for its underlying set $\R \cup \{ \infty \}$.  Multiplication in $\T$ is given by addition of (extended) real numbers, and hyperaddition is defined as follows: if $a \neq b$ then $a\hyperplus b = \min(a,b)$, while $a\hyperplus a = \{ c \in \R \; : \; c \geq a \} \cup \{ \infty \}$.  The hyperinverse of $x$ is equal to $x$ for all $x \in \T$.
\end{itemize}

The following hyperfields will also be used later on to give some examples and counterexamples:

\begin{itemize}
 \item The \emph{Krasner hyperfield} $\K$ consists of two elements $\{ 0,1 \}$ with the usual multiplication and hyperaddition characterized by the rule $1\hyperplus 1=\{0, 1\}$.
 \item The \emph{weak sign hyperfield} $\W$ consists of three elements $\{0,1,-1\}$ with the usual multiplication and hyperaddition characterized by the rules $1\hyperplus 1=-1\hyperplus -1 = \{1, -1\}$ and $1\hyperplus-1=\{0,1,-1\}$. 
 \item The \emph{phase hyperfield} $\P$ has for its underlying set $S^1 \cup \{ 0 \}$, where $S^1 = \{ e^{i \theta} \in {\mathbb C} \; | \; 0 \leq \theta < 2\pi \}$ is the complex unit circle. Multiplication on $\P$ is deduced from the multiplication on $\C$, and hyperaddition is characterized by the following rules:
   \begin{itemize}
   \item If $\theta_1 = \theta_2 + \pi$, then $e^{i \theta_1} \hyperplus e^{i \theta_2} = \{ 0, e^{i \theta_1}, e^{i \theta_2} \}$. 
   \item If $\theta_1 < \theta_2 < \theta_1 + \pi$, then $e^{i \theta_1} \hyperplus e^{i \theta_2} = \{ e^{i \theta} \; | \; \theta_1 < \theta < \theta_2 \}$. 
   \end{itemize}
\end{itemize}

\begin{rem} \label{rem:inducedhyperaddition}
All six of these examples are special cases of a general construction of hyperfields as quotients of fields by a multiplicative subgroup, which is described in \cite{Connes-Consani11}.  Let $K$ be a field and let $G$ be a subgroup of $K^\times$. Then the quotient $K/G$ of $K$ by the action of $G$ by (left) multiplication carries a natural structure of a hyperfield: we have $(K/G)^\times=K^\times/G$ as an abelian group and 
 \[
  [a]\hyperplus[b] \ = \ \big\{ \ [c] \ \big| \ c=a'+b'\text{ for some }a'\in[a], b'\in[b] \ \big\} 
 \]
 for equivalence classes $[a]$ and $[b]$ in $K/G$.  
 
 For any field $K$ we have $K = K / \{ 1 \}$ and if $|K|>2$ then $\K = K / K^\times$.
 Similarly, $\S = \R / \R_{>0}$, $\P = \C / \R_{>0}$, and $\W = \F_p / (\F_p^\times)^2$ for any prime $p \geq 7$ with $p \equiv 3 \pmod{4}$.
 The tropical hyperfield $\T$ is also a special case of the quotient construction: if $K$ is any field endowed with a surjective valuation $v : K^\times \to \R$, then $\T = K / v^{-1}(0)$.
\end{rem}

\begin{rem}
 There are examples of hyperfields which do not arise from the construction given in Remark~\ref{rem:inducedhyperaddition}; see, for example, \cite{Massouros85}.
\end{rem}

\subsection*{Roots and multiplicities}

If $p(T)=\sum_{i=0}^n c_iT^i$ is a polynomial with coefficients in a field $K$, an element $a \in K$ is a root of $p$ if and only if either of the following two equivalent conditions is satisfied:
\begin{enumerate}
 \item\label{fieldroot1} $p(a)=0$, i.e., $\sum c_ia^i = 0$.
  \item\label{fieldroot2} $T-a$ divides $p(T)$, i.e., there is a polynomial $q(T)=\sum_{i=0}^{n-1} d_iT^i \in K[T]$ such that $p(T)=(T-a)q(T)$.  
  \end{enumerate}
  
Note that \eqref{fieldroot2} is equivalent to the existence of elements $d_0,\ldots,d_{n-1} \in K$ such that
       \begin{equation}\label{eq:fieldroot}
        c_0=-ad_0, \; \; c_i =  -ad_i + d_{i-1} \; \text{for} \; i=1,\dots, n-1, \; \; \text{and} \; c_{n}=d_{n-1}. \tag{$2'$} 
       \end{equation}

If $F$ is a hyperfield, then in order to define what it means to be a root of a polynomial over $F$ we will generalize conditions \eqref{fieldroot1} and \eqref{eq:fieldroot} by replacing sums with hypersums.  

\begin{lemA}\label{lem: roots}
Let $c_0,\ldots,c_n \in F$.  The following are equivalent for an element $a \in F$:
\begin{enumerate}
 \item\label{root1} $0\in \hypersum c_ia^i$.
 \item\label{root2} There exist elements $d_0,\ldots,d_{n-1} \in F$ such that
        \[
        c_0=-ad_0, \; \; c_i \in  (-ad_i) \hyperplus d_{i-1} \; \text{for} \; i=1,\dots, n-1, \; \; \text{and} \; c_{n}=d_{n-1}.
       \]
\end{enumerate}
\end{lemA}

We write $0 \in p(a)$ if \eqref{root1} is satisfied, and $p \in (T-a)q$ if $q=\sum_{i=0}^{n-1} d_iT^i$ satisfies \eqref{root2}.  

We will give a proof of Lemma~\ref{lem: roots} in section~\ref{section: hyperfields}.

\begin{rem}
 Note that, unlike the case where $F=K$ is a field, the ``quotient'' polynomial $q=\sum_{i=0}^{n-1} d_iT^i$ is in general not unique.
 For example, suppose $F=\S$ and let $p(T)=T^3-T^2-T+1$.  Then $p \in (T-1)q$ for $q(T) \in \{ T^2-1,T^2+T-1, T^2-T-1 \}$.
\end{rem}

Lemma~\ref{lem: roots} motivates the following definition:

\begin{df}\label{def: roots}
 Let $c_0,\ldots,c_n \in F$.  An element $a\in F$ is a \emph{root} of the polynomial $p=\sum_{i=0}^n c_iT^i$ if it satisfies either of the equivalent conditions \eqref{root1} or \eqref{root2}.
\end{df}

We define the \emph{multiplicity $\mult_a(p)$ of $a$} as a root of $p$ in terms of a simple recursion as follows. 

\begin{df} \label{df:multiplicity}
If $a$ is not a root of $p$, set $\mult_a(p)=0$. If $a$ is a root of $p$, define
\begin{equation*} \label{eq:multiplicity}
 \mult_a(p) \ = \ 1+\max\big\{ \, \mult_a(q) \, \big| \, p \in (T-a)q\, \big\}.
\end{equation*}
\end{df}

Note that when $F=K$ is a field, $\mult_a(p)$ is just the usual multiplicity of $a$ as a root of $p$.  

\begin{rem}
The idea to define roots of polynomials over hyperfields using \eqref{root1} is due to Viro, cf.\ \cite{Viro11}.
However, we believe that Lemma~\ref{lem: roots} and the definition of $\mult_a(p)$ in Definition~\ref{df:multiplicity} are new to this paper.
\end{rem}

\subsection*{Homomorphisms of hyperfields}

\begin{df}
Let $F_1,F_2$ be hyperfields.  A map $f:F_1 \to F_2$ is called a \emph{hyperfield homomorphism} if $f(0)=0$, $f(1)=1$, $f(ab)=f(a)f(b)$, and $f(a+b)\subset f(a)\hyperplus f(b)$ for all $a,b \in F_1$. 
\end{df}

\begin{ex}
Here are a couple of examples of hyperfield homomorphisms.
\begin{enumerate}
\item The function ${\rm sign} : \R \to \S$ taking a real number to its sign is a homomorphism of hyperfields.  
\item If $K$ is a field, a map $v : K\to \R \cup \{ \infty \}$ is called a (Krull) \emph{valuation} if $v^{-1}(\infty) = 0$ and for all $a,b \in K$ we have $v(ab)=v(a)+v(b)$ and $v(a+b) \geq 
\min \{ v(a),v(b) \}$.
One checks easily that a map $v:K\to\R\cup\{\infty\}$ is a valuation if and only if the corresponding map $K \to \T$ is a homomorphism of hyperfields.
\end{enumerate}
\end{ex}

\begin{propA}\label{propA}
 Let $K$ be a field and $f:K\to F$ a homomorphism to a hyperfield $F$.
 Let $p=\sum c_iT^i$ be a polynomial over $K$ and let $\bar{p}=\sum f(c_i)T^i$ the corresponding polynomial over $F$. Then 
 \begin{equation} \label{eq:eqA}
  \mult_b(\bar{p}) \ \geq \ \sum_{a \in f^{-1}(b)} \mult_a(p)
 \end{equation}
 for every $b\in F$.
 Moreover, if $\sum_{b \in F} \mult_b(\bar{p}) \leq {\rm deg}(\bar{p})$ and $p$ splits into a product of linear factors over $K$, then we have equality in \eqref{eq:eqA}.
\end{propA}

We will give a proof of Proposition~\ref{propA} in section~\ref{section: hyperfields}.

\begin{rem}[A pathological example]
 If $F$ is a hyperfield and $p$ is a polynomial of degree $d$ over $F$, it is possible for the sum $\sum_{a \in F} \mult_a(p)$ to exceed $d$.
 For example, if $F=\W$ is the weak sign hyperfield, then both $1$ and $-1$ are double roots of the quadratic polynomial $p(T)=T^2 + T + 1$.  
 (Indeed, it is immediately verified that $0 \in q(1)$ for $q \in \{ p, T-1 \}$, $0 \in q(-1)$ for $q \in \{ p, T+1 \}$, $p \in (T+1)(T+1)$, and $p \in (T-1)(T-1)$.)

 Such ``pathological'' behavior does not happen when $F$ is a field or when $F=\K, \; \S,$ or $\T$; in these cases, $\sum_{a \in F} \mult_a(p) \leq d$ for every polynomial $p$ over $F$ by Remarks \ref{rem:Krasner}, \ref{rem:degD}, and \ref{rem:degN} below.
\end{rem}

\begin{rem}[An even more pathological example] \label{rem:infiniteroots}
 A nonzero polynomial $p$ over a hyperfield $F$ can have {\em infinitely many} roots, in which case $\sum_{a \in F} \mult_a(p) = \infty$. For example, take $F=\P$ to be the phase hyperfield and let $p(T)=T^2 + T + 1$.  Then $a=e^{i \theta}$ is a root of $p$ for all $\pi/2 < \theta < 3\pi/2$.
%
\end{rem}

\begin{rem} \label{rem:Krasner}
 If $p(T)=c_r T^r + c_{r+1}T^{r+1} + \cdots + c_n T^n$ is a polyomial over the Krasner hyperfield $\K$, where we assume that $c_r, c_n \neq 0$, then one checks easily that $\mult_0(p)=r$ and $\mult_1(p)=n-r$.
\end{rem}

\begin{rem}
The inequality provided by Proposition \ref{propA} does not hold in general if we replace $K$ by an arbitrary hyperfield. For example, consider the map $f:\P\to\K$ from the phase hyperfield to the Krasner hyperfield that sends all nonzero elements of $\P$ to $1$. By Remark \ref{rem:Krasner}, $1$ is a root of multiplicity $2$ of the polynomial $T^2+T+1$ over $\K$. But when considered as a polynomial over $\P$, it has infinitely many roots by Remark \ref{rem:infiniteroots}, so the sum of the multiplicities of all roots of $T^2+T+1$ in $\P$ is not bound from above by $2$. 
 
 It follows from the results of this text, however, that Proposition \ref{propA} is true for $\T$ or $\S$ in place of $K$. In the case of the tropical hyperfield $\T$, this follows from the unique factorization of polynomials over $\T$ into linear terms (Theorem \ref{thm: fundamental theorem for the tropical hyperfield}), which allows for the same arguments as in the case of a field to prove the inequality of Proposition \ref{propA}. In the case of the sign hyperfield $\S$, the inequality of Proposition \ref{propA} can be established along the following lines. There is only one proper surjection from $\S$ to another hyperfield, namely the map $f:\S\to\K$ that sends $\pm 1$ to $1$. Let $p=\pm T^n+\dotsc\pm T^m$ be a polynomial over $\S$ and $\bar p$ its image polynomial over $\K$. The inequality for $b=0$ is clear. Since $\mult_{-1}(p)=\mult_1(p(-T))$, we can deduce the inequality $\mult_1(p)+\mult_{-1}(p)\leq n-m$ from Descartes's rule of signs (Theorem \ref{thmD}). Since $n-m=\mult_1(\bar p)$ (cf.\ Remark \ref{rem:Krasner}), we obtain the desired inequality.
\end{rem}

\subsection*{Multiplicities over the sign hyperfield and Descartes' rule of signs}

Let $p(T)=\sum c_iT^i$ be a polynomial over the sign hyperfield $\S$, so that all coefficients are $0$, $1$ or $-1$.  We define the \emph{number of sign changes in the coefficients of $p$} as
\[
 \sigma(p) \ = \ \# \ \big\{\, i \, \big| \, c_i=-c_{i+k}\neq 0\text{ and }c_{i+1}=\dotsb=c_{i+k-1}=0\text{ for some }k\geq1 \, \big\}.
\]

The following result will be proved in section~\ref{section: multiplicities over the sign hyperfield}.

\begin{thmA}\label{thmD}
 Let $p$ be a polynomial over $\S$. Then $\mult_1(p)=\sigma(p)$.
\end{thmA}

\begin{rem} \label{rem:degD}
 We leave it as an easy exercise for the reader to verify, using Theorem \ref{thmD} and the fact that $-1$ is a root of $p(T)$ if and only if $1$ is a root of $p(-T)$, that if $p$ is a polynomial over $\S$ then $\sum_{a \in \S} \mult_a(p) \leq {\rm deg}(p)$.
\end{rem}

As a consequence of Theorem \ref{thmD} and Proposition \ref{propA}, we obtain a new proof of Descartes' rule of signs.

\begin{thm*}[Descartes' rule of signs]
 Let $p=\sum c_iT^i$ be a polynomial over $\R$ and let $\overline{p}=\sum \sign(c_i)T^i$.
 Then the number of positive real roots of $p$ (counting multiplicities) is at most $\sigma\big(\overline{p}\big)$, with equality if $p$ splits into a product of linear factors over $\R$.  
 \end{thm*}

\begin{proof}
 Since neither $\sum_{\substack{a>0}} \mult_a(p)$ nor $\sigma\big(\overline{p}\big)$ changes if we multiply $f$ by a nonzero real number, we can assume that $f$ is monic. By Theorem \ref{thmD}, $\sigma\big(\overline{p}\big)=\mult_1\big(\overline{p}\big)$. Since $\sign(a)=1$ if and only if $a>0$, Proposition \ref{propA} implies
 \[
  \sum_{\substack{a>0}} \mult_a(p) \ \leq \ \mult_1\big(\overline{p}\big) \ = \ \sigma\big(\overline{p}\big),
 \]
 which establishes the first part of the theorem.  The assertion regarding equality when $p$ splits into a product of linear factors over $\R$ follows from Proposition \ref{propA} and Remark~\ref{rem:degD}.  
  \end{proof}
 
\begin{rem} 
 For any polynomial $\bar{p} \in \S[T]$ there exists a polynomial $p \in \R[T]$ with ${\rm sign}(p)=\bar{p}$ such that the number of positive (resp. negative) real roots of $p$ (counting multiplicities) is equal to $\mult_1(\bar{p})$ (resp. $\mult_{-1}(\bar{p})$), cf.\ \cite{Grabiner99}. So the bound given by our Proposition \ref{propA} when the homomorphism in question is ${\rm sign} : \R \to \S$ is tight in a particularly strong sense.  

 It would be interesting to characterize the hyperfield homomorphisms $f: K \to F$ with the property that for any $\bar{p}$ over $F$, there exists a polynomial $p \in K[T]$ with $f(p)=\bar{p}$ such that $\mult_b(\bar{p}) \ = \ \sum_{a \in f^{-1}(b)} \mult_a(p)$ for every $b\in F$.
\end{rem}

\subsubsection*{Short historical account}

Descartes stated his rule without proof in the appendix La g\'eom\'etrie (\cite{Descartes}) to his book {D}iscours de la m\'ethode, which was published in 1637. Newton restated this formula in 1707, also without a proof. The first proof appears in 1740 in text Usages de l'analyse de Descartes (\cite{Gua-de-Malves}) by de Gua de Malves. It was reproven by Gau\ss\ (\cite{Gauss}) in 1728, including the addition that the difference between the number of positive real roots and number of sign changes is always even, which is often mentioned as a part of Descartes' rule.


\subsection*{Tropical multiplicities and Newton's polygon rule} 


\begin{df} \label{def:NP}
Given a polynomial $p = \sum_{i=0}^n c_i T^i$ of degree $n$ with $c_i \in \T$, its {\em Newton polygon} ${\rm NP}(p)$ is defined to be the lower convex hull of $\{ (i, c_i) \; : \; 0 \leq i \leq n \} \subset \R^2$.
(For simplicity, we assume that $c_0 \neq \infty$; this allows us to avoid having to consider vertical segments in the Newton polygon.)
\end{df}

More vividly,  imagine the points $(i,c_i)$ as nails sticking out from the plane and attach a long piece of string with one end nailed to $(x_0,y_0)=(0,c_0)$ and the other end free.  Rotate the string counter-clockwise until it meets one of the nails; this will be the next vertex $(x_1,y_1)$ of the Newton polygon.   As we continue rotating, the segment $L_1$ of string between $(x_0,y_0)$ and $(x_1,y_1)$ will be fixed.  Continuing to rotate the string in this manner until the string catches on the point $(x_t,y_t) = (n,c_n)$ yields the Newton polygon of $p$.

Thus ${\rm NP}(f)$ is a finite union $L_1,\ldots,L_t$ of line segments, each with a different slope.  We let $s_j$ be the negative of the slope of $L_j$ and we denote by $\lambda_j$ the length of the projection of $L_j$ to the $x$-axis.

Finally, for $s \in \R$ we define $\nu_s(p)$ to be 0 if $s \neq s_j$ for all $j=1,\ldots,t$, and otherwise we set $\nu_s(p) = \lambda_j$, where $L_j$ is the unique segment of ${\rm NP}(f)$ with $s_j = s$.

\begin{ex}
 We illustrate these definitions in the following example. 
 Let $p=\sum c_iT^i$ be the monic polynomial of degree $5$ with $c_0=2$, $c_1=0$, $c_2=1$, $c_3=\infty$, $c_4=-1$ and $c_5=0$. 
 Then the Newton polygon can be illustrated as follows:
 \[
  \beginpgfgraphicnamed{tikz/fig1}
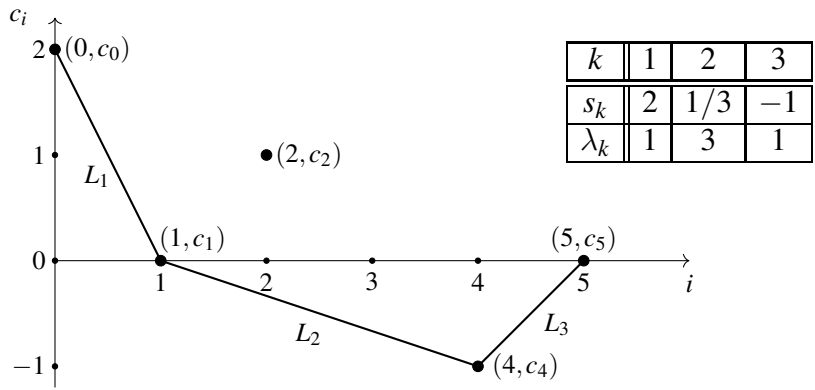
  \begin{tikzpicture}[inner sep=0,x=40pt,y=40pt,font=\footnotesize]
   \filldraw (0,0) circle (1pt);
   \filldraw (1,0) circle (1pt);
   \filldraw (2,0) circle (1pt);
   \filldraw (3,0) circle (1pt);
   \filldraw (4,0) circle (1pt);
   \filldraw (5,0) circle (1pt);
   \filldraw (0,-1) circle (1pt);
   \filldraw (0,2) circle (1pt);
   \filldraw (0,1) circle (1pt);
   \filldraw (0,2) circle (1pt);
   \node at (-0.15,0) {$0$};
   \node at (-0.15,1) {$1$};
   \node at (-0.15,2) {$2$};
   \node at (-0.25,-1) {$-1$};
   \node at (1.0,-0.2) {$1$};
   \node at (2.0,-0.2) {$2$};
   \node at (3.0,-0.2) {$3$};
   \node at (4.0,-0.2) {$4$};
   \node at (5.0,-0.2) {$5$};
   \filldraw (0,2) circle (2pt);
   \filldraw (1,0) circle (2pt);
   \filldraw (2,1) circle (2pt);
   \filldraw (4,-1) circle (2pt);
   \filldraw (5,0) circle (2pt);
   \node at (0.4,2) {$(0,c_0)$};
   \node at (1.3,0.2) {$(1,c_1)$};
   \node at (2.4,1) {$(2,c_2)$};
   \node at (4.45,-1) {$(4,c_4)$};
   \node at (5.0,0.2) {$(5,c_5)$};
   \draw[-, thick] (0,2) -- (1,0) -- (4,-1) -- (5,0);
   \node at (0.4,0.8) {$L_1$};
   \node at (2.4,-0.7) {$L_2$};
   \node at (4.75,-0.6) {$L_3$};
   \draw[->] (0,0) -- (6,0) node[below=5pt] {$i$};
   \draw[->] (0,-1.2) -- (0,2.3) node[left=10pt] {$c_i$};
   \node[font=\normalsize] at (6.0,1.5) {
   \begin{tabular}{|c||c|c|c|}
    \hline
    $k$   & $1$ & $2$ & $3$ \\
    \hline \hline
    $s_k$ & $2$ & $1/3$ & $-1$ \\
    \hline
    $\lambda_k$ & $1$ & $3$ & $1$ \\
    \hline
   \end{tabular}
   };
  \end{tikzpicture}
  \endpgfgraphicnamed
 \]
 We display the values of the $s_k$ and the $\lambda_k$ for the line segments $L_1$, $L_2$ and $L_3$ in the table next to the graphic. Thus the function $\nu_s(p)$ has the values $\nu_{2}(p)=1, \nu_{1/3}(p)=3, \nu_{-1}(p)=1$, and $\nu_s(p)=0$ for all $s \not\in \{ 2, 1/3, -1 \}$.
\end{ex}

The following result will be proved in section~\ref{section: multiplicities of tropical roots}.

\begin{thmA} \label{thmN}
Let $p$ be a polynomial over $\T$. For every $s \in \T$, we have $\mult_s(p) = \nu_s(p)$.
\end{thmA}

\begin{rem} \label{rem:degN}
It follows immediately from Theorem \ref{thmN} that $\sum_{a \in \T} \mult_a(p) = {\rm deg}(p)$ for every polynomial $p$ over $\T$.
\end{rem}

Using Theorem~\ref{thmN} and Proposition \ref{propA}, we deduce:

\begin{thm*}[Newton's polygon rule]
Let $K$ be a field and let $v : K \to \T$ be a valuation.  Let $p=\sum c_iT^i$ be a polynomial over $K$ and let $\overline{p}=\sum v(c_i)T^i$. Let $s \in \T$.  Then the number of roots $a \in K$ of $p$ with $v(a)=s$ (counting multiplicities) is at most $\nu_s(\overline{p})$, with equality if $p$ splits into a product of linear factors over $K$.
\end{thm*}

\begin{proof}
 Newton's polygon rule can be proven with the same argument as Descartes' rule of signs, where we rely on Theorem \ref{thmN} instead of Theorem \ref{thmD} in this case. Namely, by Proposition \ref{propA} and Theorem \ref{thmN}, we have
 \[
  \sum_{a\in v^{-1}(s)} \mult_a(p) \ \leq \ \mult_s(\overline p) \ = \ v_s(\overline p).
 \]
 Thus the first claim of the theorem. If $p$ splits into linear factors, then we deduce from this inequality and Remark \ref{rem:degN} that
 \[
  \deg(p) \ = \ \sum_{a\in K}\, \mult_a(p) \ \leq \ \mult_s(\overline p) \ = \ \deg(\overline p) \ = \ \deg(p),
 \]
 and thus equality throughout, which establishes the second claim of the theorem.
\end{proof}

\begin{rem}
 If $K$ is complete with respect to the valuation $v$ (i.e., $K$ is complete as a metric space with respect to the distance function $d(a,b) = e^{-v(a-b)}$), then $v$ extends uniquely to a valuation on any fixed algebraic closure $\bar{K}$ of $K$; cf.\ \cite[Chapter II, Thm.\ 4.8]{Neukirch99}. So in this case, Newton's polygon rule can be formulated as follows: the number of roots $a \in \bar{K}$ of $p$ with $v(a)=s$ (counting multiplicities) is equal to $\nu_s(\overline{p})$.
\end{rem}

\begin{rem}
 When $K$ is complete, one often uses Hensel's Lemma \cite[Chapter II, Lemma 4.6]{Neukirch99} in conjunction with Newton's polygon rule to guarantee the existence of precisely $\nu_s(\overline{p})$ roots in $K$ with valuation $s$.  For example, if $p$ has coefficients in the valuation ring $R$ of $K$ and the reduction of $p$ modulo the maximal ideal of $R$ splits completely into distinct linear factors, then it follows from Hensel's Lemma that $p$ splits completely into linear factors over $K$.
\end{rem}

\begin{rem}It would be interesting to find other useful applications of Proposition \ref{propA} besides Descartes' rule and Newton's polygon rule.\footnote{Note added: the first author's student Trevor Gunn has recently found a simultaneous generalization of Descartes' rule and Newton's polygon rule by applying Proposition \ref{propA} to the {\em signed tropical hyperfield}; see \cite{Gunn19} for details.} It would also be interesting to formulate a higher-dimensional version of the theory of multiplicities developed in this paper.
\end{rem}

\subsection*{Relation to the supertropical numbers and the symmetrization of the tropical numbers}

The tropical hyperfield is closely related to the supertropical numbers, which were introduced by Izhakian in \cite{Izhakian09}. To explain, we can extend the product and the hyperaddition on the tropical hyperfield $\T$ to the whole powerset $\cP(\T)$ of $\T$ by elementwise evaluation. With these operations, $\cP(\T)$ becomes a semiring. The smallest subsemiring of $\cP(\T)$ that contains all singletons $\{a\}$ with $a\in\T$ consists of all singletons $\{a\}$ together with all intervals of the form $[a,\infty]$, where $a$ varies through $\T$. This subsemiring is isomorphic to Izhakian's semiring of supertropical numbers. In joint work with Rowen, Izhakian extends his theory in \cite{Izhakian-Rowen10a} and \cite{Izhakian-Rowen10b} to polynomials over the semiring of supertropical numbers. The common intersections with the content of this paper are: (a) they introduce the concept of a root, which agrees with ours under the correspondence described above; and (b) they show in \cite[Lemma 5.7]{Izhakian-Rowen10a} that every supertropical polynomial has a root.

In the same way, the power set $\cP(\S)$ of the sign hyperfield $\S$ is a semiring, and the smallest subsemiring which contains all singletons is isomorphic to the symmetrization of the Boolean semifield. This semiring, or more precisely its extension to the symmetrization of the tropical numbers, are introduced in \cite{Plus90}; also cf.\ \cite[section 3.4]{Baccelli-Cohen-Olsder-Quadrat92}. In \cite[section 3.6]{Baccelli-Cohen-Olsder-Quadrat92}, polynomials over this semiring are treated, including the concepts of roots and their multiplicities. Their notion of roots coincides with ours, but their definition of multiplicity for roots is different from ours.



\section{Hyperfields}
\label{section: hyperfields}

To give a rigorous definition of hyperfields, we first define a {\em binary hyperoperation} on a set $G$ to be a map
\[
 \hyperplus: \quad G\times G \quad \longrightarrow \quad \cP(G)
\]
into the power set $\cP(G)$ of $G$ such that $a\hyperplus b$ is non-empty for all $a,b \in G$.

The hyperoperation $\hyperplus$ is called {\em commutative} if $a \hyperplus b = b \hyperplus a$ for all $a,b \in G$, and {\em associative} if 
\[
\bigcup_{d\in b\hyperplus c} a\hyperplus d = \bigcup_{d\in a\hyperplus b} d\hyperplus c
\]
for all $a,b,c \in G$.

If $\hyperplus$ is both commutative and associative, we can define the hypersum $\hypersum_{i=1}^n \ a_i$ for all $n\geq 2$ and $a_1,\ldots,a_n \in G$ by the recursive formula
\[
 \hypersum_{i=1}^n \ a_i \ = \  \bigcup_{b\in\hypersum_{i=1}^{n-1} a_i} b\hyperplus a_n.
\]

A \emph{commutative hypergroup} is a set $G$ endowed with a commutative and associative binary hyperoperation $\hyperplus$ and a distinguished element $0 \in G$ such that for all
$a,b,c \in G$:

\begin{enumerate}[label={(HG\arabic*)}]
 \item\label{HG1} $0\hyperplus a=a\hyperplus 0=\{a\}$.  \hfill\textit{(neutral element)}
 \item\label{HG2} There is a unique element $-a$ in $G$ such that $0\in a\hyperplus (-a)$.  \hfill\textit{(inverses)}
 \item\label{HG3} $a\in b\hyperplus c$ if and only if $-b \in (-a)\hyperplus c$.    \hfill\textit{(reversibility)}
\end{enumerate}

A \emph{hyperfield} is a set $F$ together with a binary operation $\cdot$, a binary hyperoperation $\hyperplus$, and distinguished elements $0$ and $1$ such that for all $a,b,c \in F$:
\begin{enumerate}[label={(HF\arabic*)}]
 \item\label{HF1} $(F,\hyperplus,0)$ is a commutative hypergroup.
 \item\label{HF2} $(F \setminus \{0 \},\cdot,1)$ is an abelian group.
 \item\label{HF3} $a\cdot 0=0\cdot a = 0$.
 \item\label{HF4} $a\cdot(b\hyperplus c)=ab\hyperplus ac$, where $a\cdot(b\hyperplus c)=\{ad \; | \; d\in b\hyperplus c\}$.   \hfill\textit{(distributivity)}
\end{enumerate}

\medskip

We illustrate the utility of the hyperfield axioms with the following proof of Lemma~\ref{lem: roots}:

\begin{proof}[Proof of Lemma~\ref{lem: roots}]
 The case $a=0$ is easy: we have $0\in \hypersum c_ia^i=0\hyperplus\dotsb\hyperplus 0\hyperplus c_0$ if and only if $c_0=0$. On the other hand, the conditions in \eqref{root2} reduce to
 \[
 c_0=0, \; \; c_i\in 0 \hyperplus d_{i-1}=\{d_{i-1}\} \; \text{for} \; i=1,\dots n-1, \; \; \text{and} \; c_{n}=d_{n-1},
 \]
 which can be fulfilled (uniquely) by $d_i=c_{i+1}$ for $i=0,\dotsc,n-1$ if and only if $c_0=0$. This establishes the desired equivalence for $a=0$.
 
 If $a\neq 0$, then by the very definition of the hypersum of $n+1$ summands, $0\in \hypersum c_ia^i$ if and only if there is a sequence of elements $e_1,\dotsc,e_{n-1} \in F$ such that 
 \[
  e_1 \ \in \ c_0\hyperplus c_1a, \quad e_i \ \in \ e_{i-1}\hyperplus c_i a^i \quad \text{for} \quad i=2,\dotsc,n-1, \quad \text{and} \quad 0\in e_{n-1} \hyperplus c_n a^n.
 \]
 Let $d_0,\dotsc,d_{n-1}\in F$ be the unique elements satisfying $c_0=-ad_0$ and $e_i=-d_ia^{i+1}$. Then the above relations can be rewritten as
 \[
  -d_i a^{i+1} \ \in \ (-d_{i-1}a^i) \hyperplus c_i a^i \quad \text{for} \quad i=1,\dotsc,n-1, \quad \text{and} \quad -d_{n-1} a^n \ = \ -c_n a^n.
 \]
(Here we use the fact that, by \ref{HG2}, $0\in e_{n-1} \hyperplus c_n a^n$ if and only if $e_{n-1}=-c_na^n$.) 

These relations can be brought into the form in which they appear in \eqref{root2} by first multiplying each of them by $-a^{-i}$ and then using the reversibility axiom \ref{HG3} to exchange the terms $d_i$ and $-c_i$. 
\end{proof}

We also give the promised proof of Proposition~\ref{propA}:

\begin{proof}[Proof of Proposition~\ref{propA}]
 Let $a_1,\dotsc,a_n\in K$ be not necessarily distinct elements such that $\prod(T-a_i)$ divides $p$ in $K[T]$. Define $q_1=p$ and for $i=1,\dotsc,n$, define the polynomial $q_{i+1} \in K[T]$ by the property that $q_i=(T- a_i) q_{i+1}$ in $K[T]$. 
 
 To prove the proposition, assume that $p(a_1)=\dotsc=p(a_n)=b$ and that there is no $a\in K$ such that $f(a)=b$ and $q_{n+1}(a)=0$, i.e., that $a_1,\dotsc,a_n$ are all of the roots of $p$ (counted with multiplicities) having $f(a_i)=b$.   
 
 By the definition of a homomorphism of hyperfields, the relations $q_i=(T-a_i)q_{i+1}$ imply that $\overline{q_i} \in (T-b) \overline{q_{i+1}}$ over $F$, where $\overline q_i$ is the image of $q_i$ under $f$. Thus the sequence of the $\overline{q_i}$ certifies that $\mult_b(p)$ is at least $n$. This proves the first part of the proposition.
 
 If $p$ splits into linear factors and $\sum_{b\in F}\mult_b(\overline p)\leq \deg\overline p$, then the first assertion of the proposition implies that 
 \[
  \deg p \ = \ \sum_{a\in K}\mult_a(p) \ \leq \ \sum_{b\in F} \mult_b(\overline p) \ \leq  \ \deg\overline p \ = \ \deg p,
 \]
 and thus equality holds throughout. Therefore $\mult_b(\bar{p})=\sum_{a \in f^{-1}(b)} \mult_a(p)$ for all $b\in F$.
\end{proof}

\section{Multiplicities over the sign hyperfield}
\label{section: multiplicities over the sign hyperfield}

Our goal in this section is to prove Theorem~\ref{thmD}.

\medskip

Let $p=\sum c_iT^i$ be a monic polynomial over the sign hyperfield $\S$ of degree $n$. Recall that the \emph{number of sign changes in the coefficients of $p$} is
\[
 \sigma(p) \ = \ \# \ \big\{\, i \, \big| \, c_i=-c_{i+k}\neq 0\text{ and }c_{i+1}=\dotsb=c_{i+k-1}=0\text{ for some }k\geq1 \, \big\}.
\]

\begin{thm}\label{thm: the multiplicity of 1 is the number of sign changes over the sign hyperfield}
 Let $p=\sum c_iT^i$ be a monic polynomial of degree $n$ over $\S$. Then $\mult_1(p)=\sigma(p)$.
\end{thm}

\begin{proof}
 The main effort of the proof consists in showing that if $\sigma(p)>0$ then
 \[
  \sigma(p) \ = \ 1+\max\big\{ \, \sigma(q) \, \big| \, p\in (T-1)q \, \big\}. 
 \]
 Once we have shown this, we can conclude the proof of the theorem by induction on $\sigma(p)$. If $\sigma(p)=0$, then $0\not\in p(1)=1\hyperplus \dotsb \hyperplus 1$ and thus $\mult_1(p)=0$. If $\sigma(p)>0$, then $0\in p(1)=c_{n}\hyperplus \dotsb \hyperplus c_0$ since there is a sign change, and
 \[
  \sigma(p) \ = \ 1+\max\{  \sigma(q) \,\big| \, p\in (T-1)q  \big\} \ = \ 1+\max\{  \mult_1(q) \, \big| \, p\in (T-1)q  \big\} \ = \ \mult_1(p),
 \]
 where we use the inductive hypothesis for the second equality and the definition of $\mult_1(p)$ for the last equality.
 
 We proceed with showing that the maximum of the values $\sigma(q)$ with $p\in (T-1)q$ is $\sigma(p)-1$. Let $q=\sum d_iT^i$ be a polynomial over $\S$ such that $p\in (T-1)q$. This means that $\deg q=\deg p-1$ and
 \[
d_0=-c_0, \; \; c_i\in -d_i\hyperplus d_{i-1} \; \text{for} \; i=1,\dotsc,n-1, \text{and} \; \; d_{n-1}=c_n=1.
 \]
 The strategy of the proof is to bound the number of sign changes in $q$ by the number of sign changes in $p$ in decreasing order of $i$.

 Let $\sigma_i(p)$ be the number of sign changes in the sequence of coefficients $c_n,\dotsc,c_i$ of $p$, i.e.,\ 
 \[
  \sigma_i(p) \ = \ \#\, \big\{k\geq i \, \big| \, c_{k}=-c_{k+l+1}\neq0\text{ and }c_{k+1}=\dotsb=c_{k+l}=0\text{ for some }l\geq0 \, \big\}.
 \]
 Let $\sigma_i(q)$ be the number of sign changes in the sequence of coefficients $d_{n-1},\dotsc,d_i$ of $q$, which is defined analogously to $\sigma_i(p)$. 
 
 We claim that $\sigma_i(q)\leq\sigma_i(p)$ for all $i=0,\dotsc,n$, with $\sigma_i(q)+1\leq\sigma_i(p)$ if $d_i=-c_i\neq0$. We will prove this claim by descending induction on $i$. If $i=n$, then $\sigma_i(q)=\sigma_i(p)=0$, which proves our claim in this case since $d_n=0\neq -c_n$.
 
 Before explaining the inductive step, we begin with some preliminary observations which allow us to simplify the situation and limit the number of cases that we have to consider. Namely, if $0\neq c_i$ and $c_i\neq -c_{i+1}$ as well as $0\neq d_i$ and $d_i\neq -d_{i+1}$, then we have $\sigma_i(p)=\sigma_{i+1}(p)$ and $\sigma_i(q)=\sigma_{i+1}(q)$. Thus we do not change the values of $\sigma_i(p)$ and $\sigma_i(q)$ if we omit $c_{i+1}$ and $d_{i+1}$ from the sequences $c_n,\dotsc,c_i$ and $d_{n-1},\dotsc,d_i$. Therefore we may assume without loss of generality that this situation does not occur. We may similarly assume that $c_0\neq0$, since otherwise $d_0=-c_0=0$ and thus $\sigma_0(p)=\sigma_{1}(p)$ and $\sigma_0(q)=\sigma_{1}(q)$. 
 
 These assumptions and the relation $p\in(T-1)q$ have the following consequences for $i=0,\dotsc,n-1$:
 \begin{enumerate}
  \item\label{pre1} We have $c_{i+1}\neq 0$. Indeed, if $c_{i+1}=0$, then $c_{i+1}\in-d_{i+1}\hyperplus d_i$ implies that $d_{i+1}=d_{i}$. But this situation is excluded by our assumptions.
  \item\label{pre2} If $d_{i+1}=-d_i$, then $c_{i+1}\in-d_{i+1}\hyperplus d_i$ implies that $c_{i+1}=d_i=-d_{i+1}$.
  \item\label{pre3} If $c_i=-d_i$, then we have $c_{i+1}=d_i=-c_i$. Indeed, if $c_{i+1}=c_i$ then $c_{i+1}\in-d_{i+1}\hyperplus d_i$ implies $d_{i+1}=d_i$, which is excluded by our assumptions.
 \end{enumerate}

 Assume that $i<n$. We prove the inductive step of our claim by considering the following four constellations of possible values for $c_i$, $d_i$, and $d_{i+1}$. (We indicate usage of the inductive hypothesis in the following relations by ``(IH)''.)
 
 \medskip\noindent\textbf{Case 1:} {$d_{i+1}\neq -d_i$ and $c_i\neq -d_i$.}  In this case, we obtain
 \[
  \sigma_i(q) \ = \ \sigma_{i+1}(q) \ \underset{(IH)}\leq \ \sigma_{i+1}(p) \ \leq \ \sigma_i(p).
 \]
 
 \medskip\noindent\textbf{Case 2:} {$d_{i+1}=-d_i$ and $c_i\neq -d_i$.} By \eqref{pre1} and \eqref{pre2}, we have $c_{i+1}=-d_{i+1}=d_i=c_i$, and thus 
 \[
  \sigma_i(q) \ = \ \sigma_{i+1}(q)+1 \ \underset{(IH)}\leq \ \sigma_{i+1}(p) \ = \ \sigma_i(p).
 \]
 
 \medskip\noindent\textbf{Case 3:} {$d_{i+1}\neq -d_i$ and $c_i=-d_i$.} By \eqref{pre3}, we have $c_{i+1}=d_i=-c_i$, and thus
 \[
  \sigma_i(q)+1 \ = \ \sigma_{i+1}(q)+1 \ \underset{(IH)}\leq \ \sigma_{i+1}(p)+1 \ = \ \sigma_i(p).
 \]

 \medskip\noindent\textbf{Case 4:} {$d_{i+1}=-d_i$ and $c_i=-d_i$.} By \eqref{pre3}, we have $c_{i+1}=d_i=-c_i=-d_{i+1}$, and thus
 \[
  \sigma_i(q)+1 \ = \ \sigma_{i+1}(q)+2 \ \underset{(IH)}\leq \ \sigma_{i+1}(p)+1 \ = \ \sigma_i(p).
 \]
 This concludes the proof of our claim. 
 
 Note that $\sigma(p)=\sigma_0(p)$ and $\sigma(q)=\sigma_0(q)$. Since $d_0=-c_0$ and $q$ was chosen arbitrarily with respect to the property $p\in (T-1)q$, this shows that 
 \[
  \sigma(p) \ \geq \ 1+\max\big\{ \, \sigma(q) \, \big| \, p\in (T-1)q \, \big\}. 
 \]

 To complete the proof of the theorem, we have to show that there is a $q_0$ with $p\in (T-1)q_0$ and $\sigma(q_0)+1=\sigma(p)$. We define $q_0=\sum d_iT^i$ as follows. Let $k$ be the number such that $c_0=\dotsc=c_{k}=-c_{k+1}$, and define
 \begin{align*}
  d_i &= c_{i+1}   && \text{if }c_{i+1}\neq 0\text{ and }i>k; \\
  d_i &= d_{i+1}   && \text{if }c_{i+1}=0\text{ and }i>k; \\
  d_i &= -c_{0}   && \text{if }i\leq k.
 \end{align*}
 We leave the easy verification that  $p\in (T-1)q_0$ and $\sigma(q_0)+1=\sigma(p)$ to the reader.
\end{proof}


\section{Multiplicities of tropical roots} 
\label{section: multiplicities of tropical roots}

Our goal in this section is to prove Theorem~\ref{thmN}. 
Our proof is based on a hyperfield version (Theorem~\ref{thm: fundamental theorem for the tropical hyperfield} below) of the so-called ``Fundamental theorem of tropical algebra'' (cf.~Lemma \ref{lemma: fundamental theorem of tropical algebra}).


Let $p=\sum c_iT^i$ be a monic polynomial of degree $n$ over $\T$ and let $a_1,\dotsc,a_n\in\T$. We write $p\in \prod(T+a_i)$ if
\[
 c_{n-i} \in \underset{e_1<\dotsb<e_{i}}\hypersum a_{e_1}\dotsb a_{e_i}
\]
for all $i=1,\dotsc,n$.

\begin{thm}[Fundamental theorem for the tropical hyperfield]\label{thm: fundamental theorem for the tropical hyperfield}
 Let $p=\sum_{i=0}^n c_i T^i$ be a monic polynomial of degree $n$ over $\T$. Then: 
 \begin{enumerate}
 \item\label{fund1} There is a unique sequence $a_1,\dotsc,a_n\in\T$, up to permutation of the indices, such that $p\in \prod(T+a_i)$.
 \item\label{fund2} For every $a\in\T$, we have equalities
                    \[
                     \mult_a(p) \ = \ \#\big\{ \, i\in\{1,\dotsc,n\} \, \big| \, a=a_i \, \big\} \ = \ v_p(a).
                    \]
 \end{enumerate}
\end{thm}

The rest of this section is devoted to the proof of Theorem \ref{thm: fundamental theorem for the tropical hyperfield}. The main idea of the proof is to consider polynomials over the tropical hyperfield $\T$ as functions from the \textit{tropical semifield} $\overline \R$ to itself, and to compare the hyperfield and semifield perspectives. 

As a set, $\overline\R=\R\cup\{\infty\}$ is equal to $\T$, and they have the same multiplication as well: the product $ab \in \overline\R$ is defined as the sum of the corresponding extended real numbers. The difference between $\overline\R$ and $\T$ appears in the addition law: the sum of two elements $a$ and $b$ of $\overline\R$ is defined as $\min\{a,b\}$, which is an \textit{element} of $\overline\R$, opposed to the \textit{subset} $a\hyperplus b$ of $\T$. 

To avoid confusion between tropical addition and usual addition (i.e.,\ tropical multiplication!), we will adhere strictly to the following conventions. We denote elements of $\T$ by $a,b,c,d$ and elements of $\overline\R$ by $\overline a,\overline b,\overline c,\overline d$. Given an element $a\in \T$, we write $\overline a$ if we consider it as an element of $\overline\R$. We keep the previously established notations for $\T$, i.e.\ the hypersum of $a$ and $b$ is denoted by $a\hyperplus b$ and their product by $ab$. We denote the tropical sum of two elements $\overline a$ and $\overline b$ of $\overline\R$ by $\min\{\overline a,\overline b\}$ and their tropical product by $\overline a+\overline b$. We write $i\cdot \overline a$ for the $i$-fold sum $\overline a+\dotsb+\overline a$ of $\overline a$ with itself. 

A nontrivial polynomial $p=\sum c_iT^i$ of degree $n$ over $\T$ defines a function
\[
 \begin{array}{cccc}
  \overline p: & \R & \longrightarrow & \R \\
               & \overline b  & \longmapsto     & \min\limits_{i=0,\dotsc,n} \big\{\, \overline c_i + i \cdot \overline b \, \big\},
 \end{array}
\]
which we sometimes extend to a function $\overline\R\to\overline\R$ via $\overline p(\infty)=\infty$. The trivial polynomial yields the \emph{trivial function} $\overline b\mapsto\infty$. 

We say that two polynomials $p=\sum c_i T^i$ and $q=\sum d_iT^i$ over $\T$ are \emph{functionally equivalent}, denoted $\overline p=\overline q$, if they define the same function $\overline\R\to\overline\R$. We call a function $\overline p:\overline\R\to\overline\R$ as above a \emph{tropical (polynomial) function} and denote it by $\overline p=\min\{\overline c_i+i\cdot T\}$. The \emph{degree of $\overline p$} is the degree of $p$ and $\overline p$ is \emph{monic} if $p$ is monic. Note that both notions are independent of the choice of the representing polynomial $p$. Note further that the set of tropical functions inherits the structure of a semiring from $\overline\R$ by adding and multiplying functions valuewise. 

It is well-known that every tropical function factors uniquely into a product of linear functions. This result is sometimes referred to as the ``fundamental theorem of tropical algebra'', and it was first proven in \cite[Thm.\ 11]{Cunninghame-Green-Meijer80}; see also \ \cite[Thm.\ 3.43]{Baccelli-Cohen-Olsder-Quadrat92}.

\begin{lemma}[Fundamental theorem of tropical algebra]\label{lemma: fundamental theorem of tropical algebra}
 For every monic tropical function $\overline p=\min\{\overline c_i+i\cdot T\}$ of degree $n$, there is a unique sequence $\overline a_1,\dotsc,\overline a_n\in\overline\R$, up to a permutation of indices, such that $\overline p=\sum_{i=1}^n \min\{T,\overline a_i\}$ as tropical functions. \qed
\end{lemma}

The second equality in part \ref{fund2} of Theorem \ref{thm: fundamental theorem for the tropical hyperfield} follows from the usual arguments in the theory of Newton polygons; in particular, we have the following well-known fact (see \cite{Casselmann18} or \cite[section 9]{Cunninghame-Green-Meijer80}, for example, for proofs):

\begin{lemma}\label{lemma: roots of a tropical polynomials as slopes of the Newton polygon}
 Let $p=\sum c_iT^i$ be a monic polynomial of degree $n$ over $\T$ and let $a_1,\dotsc,a_n\in\T$ be such that $\overline p=\sum_{i=1}^n \min\{T,\overline a_i\}$. Let $a\in\overline\R$. Then $\#\{i \; | \; a=a_i\}=v_p(a)$. \qed
\end{lemma}

The rest of the proof of Theorem \ref{thm: fundamental theorem for the tropical hyperfield} is novel. Part \eqref{fund1} follows immediately from the following proposition, coupled with Lemma \ref{lemma: fundamental theorem of tropical algebra}.

\begin{prop}\label{prop: characterization of functional equivalence over the tropical hyperfield}
 Let $p$ be a monic polynomial of degree $n$ over $\T$ and let $a_1,\dotsc,a_n\in\T$. Then $p\in\prod(T+a_i)$ if and only if $\overline p=\sum\min\{T,\overline a_i\}$ as tropical functions.
\end{prop}

\begin{proof}
 Let $p=\sum c_iT^i$ and assume that $a_1\leq\dotsb\leq a_n$. We define
 \[
  s_i \ = \ a_1\dotsb a_i,
 \]
 which can be thought of as the $i$-th elementary symmetric polynomial evaluated at $(a_1,\dotsc,a_n)$ (with respect to the tropical addition from $\overline\R$, not the hyperaddition of $\T$). Thus
 \[
  \sum_{i=1}^n \min\{\overline b,\, \overline a_i\} \ = \ \underset{i=0,\dotsc,n}\min\{\overline s_i+i\overline b\}.
 \]
 The relation $p\in\prod(T+a_i)$ means that $c_{n-i}\geq s_i$ for all $i=1,\dotsc,n$, with equality if the minimum occurs only once among the terms $a_{e_1}\dotsb a_{e_i}$ with $1\leq e_1<\dotsb< e_i\leq n$. This is the case if and only if $a_i<a_{i+1}$.
 
 We begin with the proof that $p\in\prod(T+a_i)$ implies $\overline p=\sum\min\{T,\overline a_i\}$. For $b\in\T$, we have
 \[
  \overline p(b) \ = \ \underset{i=0,\dotsc,n}\min\{\overline c_i+i\overline b\} \ \geq \ \underset{i=0,\dotsc,n}\min\{\overline s_i+i\overline b\} \ = \ \sum_{i=1}^n \min\{\overline b,\, \overline a_i\}.
 \]
 In order to verify the reverse inequality, we choose some $a_0\leq\min\{b,a_1\}$ and define $a_{n+1}=\infty$. Then $a_k\leq b< a_{k+1}$ for some $k\in\{0,\dotsc,n\}$. Since $a_k< a_{k+1}$, we have $c_{n-k}=s_k$, as noted before. Therefore
 \[
  \overline p(\overline b) \ = \ \underset{i=0,\dotsc,n}\min\{\overline c_i+i\overline b\} \ \leq \ \overline c_{n-k}+(n-k)\overline b \ = \ \overline s_k+(n-k)\overline b \ = \ \sum_{i=1}^n \min\{\overline b,\, \overline a_i\}.
 \]
 This concludes the proof that $\overline p=\sum\min\{T,\overline a_i\}$.
 
 We continue with the reverse implication and assume that $\overline p=\sum\min\{T,\overline a_i\}$. We need to show for $k=1,\dotsc,n$ that $c_{n-k}\geq s_k$, with equality if $a_k<a_{k+1}$.
 Choose $b\in\T$ such that $a_k\leq b\leq a_{k+1}$, where we set $a_{n+1}=\infty$ as before. Then
 \[
  \underset{i=0,\dotsc,n}\min\{\overline c_i+i\overline b\}  \ = \  \overline p(\overline b) \ = \ \sum_{i=1}^n \min\{\overline b,\, \overline a_i\} \ = \ \overline a_1+\dotsb+\overline a_k+\underbrace{\overline b+\dotsb+\overline b}_{n-k\text{ times}} \ = \ \overline s_k+(n-k)\overline b.
 \]
It follows, in particular, that $c_{n-k}\leq s_k$. If $a_k< a_{k+1}$, then $\overline p(\overline b)=\overline s_k+(n-k)\overline b$ for infinitely many $b$. This is only possible if $c_{n-k}=s_k$. 
\end{proof}

We are left with proving the first equality in part \eqref{fund2} of Theorem \ref{thm: fundamental theorem for the tropical hyperfield}. As a first step, we will prove the following fact. 
(To make sense of the case $n=1$, we define the empty product of polynomials over $\T$ as $\{0\}$.)

\begin{lemma}\label{lemma: polynomial division by linear factors for the tropical hyperfield}
 Let $p$ be a polynomial over $\T$ and let $a_1,\dotsc,a_n\in\T$ be such that $p\in \prod_{i=1}^n(T+a_i)$. If $p\in (T+a_n)q$ for a polynomial $q$ over $\T$, then $q\in\prod_{i=1}^{n-1}(T+a_i)$.
\end{lemma}

\begin{proof}
 Note that the hypotheses of the proposition imply that $p$ is monic of degree $n\geq 1$ and that $q$ is monic of degree $n-1$. We prove the result by induction on $n$. If $n=1$, then $p=(T+a_1)$ and $q=0$ is contained in the empty product.
 
 Let $n>1$. By part \eqref{fund1} of Theorem \ref{thm: fundamental theorem for the tropical hyperfield}, $q\in\prod_{i=1}^{n-1}(T+a_i')$ for some sequence $a_1',\dotsc, a_{n-1}'\in\T$. This means that 
 \[
  d_i \ \in \ \underset{1\leq e_1<\dotsb<e_{n-i-1}<n}\hypersum (a'_{e_1}+\dotsb+a'_{e_{n-i-1}})
 \]
 for all $i=0,\dotsc,n-2$. Thus $p\in (T+a_n)q$ implies that 
 \begin{align*}
  c_i \ \in \ d_{i-1}\hyperplus d_ia_n \ &\subset \ \underset{1\leq e_1<\dotsb<e_{n-i}<n}\hypersum (a'_{e_1}\dotsb a'_{e_{n-i}}) \hyperplus \underset{1\leq e_1<\dotsb<e_{n-i-1}<n}\hypersum (a'_{e_1}\dotsb a'_{e_{n-i-1}}a_n) \\
                                         &= \ \underset{1\leq e_1<\dotsb<e_{n-i}\leq n}\hypersum (a'_{e_1}\dotsb a'_{e_{n-i}})
 \end{align*}
 for $i=1,\dotsc,n-1$, where we set $a_n'=a_n$. Also $c_0=d_0a_n=\prod a_i'$, and thus $p\in \prod(T+a_i')$. By the uniqueness of $a_1,\dotsc,a_n$ such that $p\in \prod(T+a_i)$ (by part \eqref{fund1} of Theorem \ref{thm: fundamental theorem for the tropical hyperfield}), we conclude that there is a permutation $\sigma\in S_{n-1}$ such that $a_i'=a_{\sigma(i)}$ for $i=1,\dotsc,n-1$. Thus $q\in\prod_{i=1}^{n-1}(T+a_i)$, as claimed.
\end{proof}

In order to complete the proof of Theorem \ref{thm: fundamental theorem for the tropical hyperfield}, consider a monic polynomial $p=\sum c_iT^i$ of degree $n$ over $\T$ with $p\in\prod(T+a_i)$ and let $a\in \T$. Then $\infty\in p(a)$ if and only if the minimum appears twice among the terms $\overline c_i+i\cdot\overline a$ for $i=0,\dotsc,n$. This means that the function $\overline p:\R\to \R$ has a change of slope at $a$, which is the case if and only if $a\in\{a_1,\dotsc,a_n\}$.

We prove that $\mult_a(p)=\#\{i \, | \, a=a_i\}$ by induction on the latter quantity. If $\#\{i \, | \, a=a_i\}=0$, then $a\notin \{a_1,\dotsc,a_n\}$ and $\infty\notin p(a)$. Thus $\mult_a(p)=0$, as desired.

If $\#\{i \, | \, a=a_i\}>0$, then $a\in \{a_1,\dotsc,a_n\}$ and $\infty\in p(a)$. After relabelling the indices, we can assume that $a=a_n$. For every polynomial $q$ over $\T$ with $p\in (T+a_n)q$, Proposition \ref{lemma: polynomial division by linear factors for the tropical hyperfield} shows that $q\in\prod_{i=1}^{n-1}(T+a_i)$. Thus the inductive hypothesis applies to $q$ and yields
\[
 \mult_a(p) \ \geq \ \mult_a(q)+1 \ = \ \#\big\{ \, i\in\{1,\dotsc,n-1\} \, \big| \, a=a_i \, \big\} +1 \ = \ \#\big\{ \, i\in\{1,\dotsc,n\} \, \big| \, a=a_i \, \big\}.
\]
By definition, $\mult_a(p)=1+\max\{\mult_a(q) \, | \, p\in (T+a)q\}$. By Lemma \ref{lem: roots}, there is a polynomial $q_0$ such that $p\in (T+a)q_0$. Since $q$ was arbitrary, the first inequality in the displayed equation is an equality. This concludes the proof of Theorem \ref{thm: fundamental theorem for the tropical hyperfield}. \hfill\qed




\appendix


\section{Polynomial algebras over hyperfields}
\label{section: polynomial algebras over hyperfields}

Up to this point, we have considered polynomials over a hyperfield $F$ as formal expressions of the form $\sum c_iT^i$ with coefficients $c_i\in F$. In this appendix, we explain how to make sense of such expressions as elements of a ``polynomial algebra'' over $F$, and how the definitions of roots and their multiplicities take a more conventional form in such a formulation.

In fact, we will consider two candidates for the polynomial algebra over a hyperfield: as a ``additive-multiplicative hyperring'' with multi-valued multiplication and addition, or as an ordered blueprint.
We argue that the second of these alternatives is the more natural and less pathological one.

\subsection{Polynomial hyperrings}
\label{subsection: polynomial hyperrings}

Let $F$ be a hyperfield. The set $\Pol(F)=\{\sum c_iT^i \, | \, c_i\in F\}$ of all polynomials over $F$ can be naturally endowed with two hyperoperations $\hyperplus$ and $\hyperdot$, which are defined for polynomials $p=\sum c_iT^i$ and $q=\sum d_iT^i$ as
\begin{align*}
 p\hyperplus q \ &= \textstyle \ \big\{ \, \sum e_iT^i \, \big| \, e_i\in c_i\hyperplus d_i \, \big\}, \\
 p\hyperdot  q \ &= \textstyle \ \big\{ \, \sum e_iT^i \, \big| \, e_i\in \underset{k+l=i}\hypersum c_kd_l \, \big\}.
\end{align*}
These operations turn $\Pol(F)$ into an additive-multiplicative hyperring which has been considered in \cite{Davvaz-Koushky04}, \cite{Janic-Rasovi07}, and other publications.

Let $a\in F$, and let $p=\sum c_iT^i$ and $q=\sum d_iT^i$ be polynomials over $F$. Then $p\in (T-a)\hyperdot q$ if and only if $n=\deg p=\deg q+1$ and 
\[
c_0=-ad_0, \; \; c_i\in (-ad_i) \hyperplus d_{i-1} \; \text{for} \; i=1,\dots, n-1, \; \text{and} \; \; c_{n}=d_{n-1}.
\]
This means that the relation $p\in (T-a)q$, as introduced in section \ref{section: statement of the main results}, is equivalent to the relation $p\in (T-a)\hyperdot q$ stemming from the hypermultiplication of polynomials over $F$.

Similar to the case of the hypersum of a hyperfield, we define $n$-fold products of polynomials over $F$ by the recursive formula
\[
 \underset{i=1}{\stackrel{n}\hyperprod} p_i \ = \ \bigcup_{q\in\hyperdot_{i=1}^{n-1} p_i} q \hyperdot p_n.
\]


In the case of the tropical hyperfield $\T$, the relation $p\in\prod(T+a_i)$ from section \ref{section: multiplicities of tropical roots} is equivalent to $p\in\hyperprod_{i=1}^n(T+a_i)$. Indeed, by multiplying out all linear terms, we find that $p\in\hyperprod_{i=1}^n(T+a_i)$ is equivalent to $p=\sum c_iT^i$ being monic of degree $n$ such that 
\[
 c_{n-i} \ \in \ \underset{1\leq e_1<\dotsb <e_i\leq n}\hypersum a_{e_1}\dotsb a_{e_i}
\]
for all $i=1,\dotsc,n$. 

In spite of these appealing interpretations of the relations $p\in(T+a)q$ and $p\in\prod(T+a_i)$, we view the (additive-multiplicative) polynomial hyperring $\Pol(F)$ as an object of limited utility due to the following two deficiencies.

\subsection{Deficiency \#1: polynomial hyperrings are not associative}
\label{subsection: polynomial hyperings are not associative}

The hypermultiplication of a polynomial hyperring fails to be associative in general. This is, for instance, the case for the polynomial algebra $\Pol(\S)$ over the sign hyperfield, as the following example (due to Ziqi Liu, cf.~\cite{Liu19}) shows:


While
\begin{align*}
 \Big((T-1)\hyperdot(T-1)\Big) \hyperdot (T+1) \ &= \ \big\{ \, T^2-T+1 \, \big\} \hyperdot (T+1) \\
                                                 &= \ \big\{ \, T^3+aT^2+bT+1 \, \big| \, a,b\in\S \,\big\},
\end{align*}
we have 
\begin{align*}
 (T-1)\hyperdot\Big((T-1) \hyperdot (T+1)\Big) \ &= \ (T-1)\hyperdot\big\{ \, T^2+aT-1 \,\big| \, a\in\S\big\} \\
                                                 &= \ \big\{ \, T^3+aT^2+bT+1 \, \big| \, a=-1\text{ or }b=-1\,\big\}.
\end{align*}
This means, in particular, that $n$-fold products $\hyperprod_{i=1}^n p_i$ of linear polynomials $p_i\in\Pol(\S)$ depend on the order of the $p_i$.

\begin{rem}
Note that this example also shows that we cannot define the multiplicities of roots in a na{\"i}ve way in terms of factorizations into linear factors: $p=T^3+T^2+T+1$ is an element of $\big((T-1)\hyperdot(T-1)\big) \hyperdot (T+1)$, but $p(1)$ does not contain $0$.
\end{rem}

\begin{rem}
Liu also shows in \cite{Liu19} that hypermultiplication in $\Pol(\T)$ is non-associative. Note, however, that Theorem \ref{thm: fundamental theorem for the tropical hyperfield} implies that the hyperproduct of \emph{linear} polynomials over $\T$ {\em is} associative, and thus $\hyperprod_{i=1}^n(T+a_i)$ is independent of the order of the factors. 
\end{rem}

\begin{rem}
 We could overcome Deficiency \#1 by extending the hyperproduct of polynomials to certain sets of polynomials in the following way. For a finite sequence $C_0,\dotsc,C_n\subset F$ of subsets of $F$, we denote by $\sum C_iT^i$ the set of polynomials $p=\sum c_iT^i$ with coefficients $c_i\in C_i$. We define
 \[\textstyle
  (\sum C_iT^i) \hypertimes (\sum D_iT^i) \ = \ \sum \Big( \hypersum_{k+l=i} C_kD_l \Big)T^i,
 \]
 which recovers the hyperproduct $(\sum c_iT^i)\hyperdot(\sum d_iT^i)=(\sum C_iT^i)\hypertimes(\sum D_iT^i)$ in the case of singletons $C_i=\{c_i\}$ and $D_i=\{d_i\}$. With these conventions, $\hypertimes$ is associative, and in particular
 \[
  \hypercross_{i=1}^n p_i \quad = \quad \Big\{ \, {\textstyle \sum} d_iT^i \, \Big| \, d_i\in\underset{j_1+\dotsb+j_n=i}\hypersum \big( \prod_{k=1}^n c_{k,j_k}\big) \,\Big\}.
 \]
 for $p_i=\sum c_{i,j}T^j$. 
 
 We will not pursue this line of thought any further; note, however, that
 $\hypertimes$ appears implicitly in our proposed solution using ordered blueprints. Namely, $q\in \hypercross_{i=1}^n p_i$ if and only if $q\leq \prod_{i=1}^n p_i$ in the associated ordered blueprint; cf.\ section \ref{subsection: polynomial hyperrings revisited}.
%
 
\end{rem}

\subsection{Deficiency \#2: polynomial hyperrings are not free}
\label{subsection: polynomial hyperrings are not free}

Polynomial hyperrings fail to satisfy the universal property of a free algebra.  In fact, it appears to be the case that neither the category of hyperrings nor a suitable category of (non-associative) additive-multiplicative hyperrings possess free algebras in general. Here we assume that a morphism of additive-multiplicative hyperrings is a map $f:R_1\to R_2$ that preserves $0$ and $1$ and satisfies $f(a\hyperplus b)\subset f(a)\hyperplus f(b)$ and $f(a\hyperdot b)\subset f(a)\hyperdot f(b)$. 

For instance, we can extend the identity map $\K\to\K$ of the Krasner hyperfield to different morphisms $\K[T]\to\K$ that map $T$ to $1$. One example is the morphism $f_1:\K[T]\to\K$ that maps every nonzero polynomial $p$ to $f(p)=1$. Another example is the morphism $f_0:\K[T]\to\K$ for which $f_0(p)=1$ if and only if $p$ is a monomial, i.e.\ $p=T^n$ for some $n\geq0$.

\subsection{Towards free algebras}
\label{subsection: towards free algebras}

One way to incorporate free (and associative) algebras over hyperfields might be to develop a theory of ``partial hyperrings'', as considered in \cite{Baker-Bowler17}, which allows for such objects. In this appendix, however, we will use the more general and already developed theory of ordered blueprints to produce free algebras which satisfy the desired universal property. We remark that one could most likely also develop a similar theory based on Rowen's notion of {\em systems} (cf.\ \cite{Rowen19}), which is similar to that of an ordered blueprint.


In layman's terms, the passage from hyperfields to ordered blueprints consists essentially in an exchange of symbols: the relations $c\in a\hyperplus b$ in a hyperfield $F$ get replaced by the relations $c\leq a+b$ in the associated ordered blueprint. Under the hood, the symbol $\leq$ refers to a partial order that is defined on the group semiring $B^+=\N[F^\times]$. 

We now outline the definition of ordered blueprints and indicate how they allow for free algebras over hyperfields; for more details, we refer the reader to \cite{Baker-Lorscheid18} and \cite{Lorscheid18}.

\subsection{Ordered blueprints}
\label{subsection: ordered blueprints}

An \emph{ordered semiring} is a commutative (and associative) semiring $R$ with $0$ and $1$ together with a partial order $\leq$ that is \emph{additive and multiplicative}, i.e.\ $a\leq b$ implies $a+c\leq b+c$ and $ac\leq bc$ for all $a,b,c\in R$. Given a set $S=\{a_i\leq b_i\}$ of relations on $R$, we say that $S$ \emph{generates the partial order on $R$} if $\leq$ is the smallest additive and multiplicative partial order of $R$ that contains $S$.

An \emph{ordered blueprint} is an ordered semiring $B^+$ together with a multiplicative subset $B^\bullet$ of $B^+$ that contains $0$ and $1$ and that generates $B^+$ as a semiring. We write $B$ for an ordered blueprint and refer to its \emph{ambient semiring} by $B^+$ and to its \emph{underlying monoid} by $B^\bullet$. 
A \emph{morphism $f:B_1\to B_2$ of ordered blueprints} is an order-preserving homomorphism $f:B_1^+\to B_2^+$ of semirings such that $f(B_1^\bullet)\subset B_2^\bullet$.

\begin{ex}\label{ex: tropical semifield}
 The tropical semifield $\overline\R$ can be considered as the ordered blueprint $B$ with $B^\bullet=B^+=\overline\R$ whose partial order satisfies $a\leq b$ if and only if $a+b=b$. 
 
 For the purpose of this appendix, we invite the reader to think of $\overline\R$ as the 
 \textit{max-times-algebra} $\R_{\geq0}$, in contrast to the \emph{min-plus-algebra} $\R\cup\{\infty\}$ used in the main part of this paper. The negative logarithm $-\log:\R_{\geq0}\to\R\cup\{\infty\}$ defines an isomorphism of semirings between these two models for $\overline\R$. Note that $\leq$ agrees with the natural order on $\R_{\geq0}$ and with the \textit{reversed} natural order on $\R\cup\{\infty\}$.
\end{ex}

\subsection{Hyperfields as ordered blueprints}
\label{subsection: hyperfields as ordered blueprints}

The incarnation of a hyperfield $F$ as an ordered blueprint $B$ is as follows. Its ambient semiring is the group semiring $B^+=\N[F^\times]$, its underlying (multiplicative) monoid is $B^\bullet=F$, and its partial order is generated by all relations of the form $c\leq a+b$ whenever $c\in a\hyperplus b$ in $F$. We illustrate this in more detail for the main examples of hyperfields which appear in this paper.

\subsubsection{Fields}\label{subsubsection: fields as ordered blueprints}
Given a field $K$, the associated hyperaddition is defined as $a\hyperplus b=\{a+b\}$. This yields the ordered blueprint $B$ with ambient semiring $B^+=\N[K^\times]$, underlying monoid $B^\bullet=K$, and partial order $\leq$ that is generated by 
\[
 c \leq a+b \qquad \text{whenever} \qquad c = a+b \qquad \text{in }K.
\]

\subsubsection{The tropical hyperfield}\label{subsubsection: the tropical hyperfield as an ordered blueprint}
As with the tropical semifield $\overline\R$, we adopt the multiplicative notation from Example \ref{ex: tropical semifield}, i.e., we identify the elements of the tropical hyperfield with $\R_{\geq0}$ and, by abuse of notation,  use the letter $\T$ for the associated ordered blueprint, which can be described explicitly as follows. The ambient semiring of $\T$ is the group semiring $\T^+=\N[\R_{>0}]$ generated by the multiplicative group of positive real numbers, the underlying monoid is $\T^\bullet=\R_{\geq0}$, and the partial order is generated by the relations 
\[
 c \leq a+b \qquad \text{whenever} \qquad c = \max\{a,b\} \qquad \text{or} \qquad c \leq a=b \qquad \text{in }\R_{\geq0}.
\]
Note that the semiring $\T^+$ is \emph{not} idempotent, in contrast to the tropical semifield $\overline\R$. Rather, it is a subsemiring of the group ring $\Z[\R_{>0}]$. The connection to $\overline\R$ is given by the identity map $\T^\bullet\to\overline\R^\bullet$ between the respective underlying monoids, which extends linearly to an order-preserving surjection $f:\T^+=\N[\R_{>0}]\to \R_{\geq0}=\overline\R^+$ of semirings, i.e.,\ $f$ is a morphism of ordered blueprints.

\subsubsection{The sign hyperfield}\label{subsubsection: the sign hyperfield as an ordered blueprint}
As an ordered blueprint, the \emph{sign hyperfield} $\S$ consists of the ambient semiring $\S^+=\N[\{1,-1\}]$, the underlying monoid $\S^\bullet=\{0,1,-1\}$, and the partial order generated by the relations
\[
 1 \leq 1+1, \qquad 1\leq 1-1, \qquad \text{and} \qquad 0\leq 1-1.
\]
Note that $0$ and $1-1$ are distinct elements in $\S^+$.

\subsection{Free algebras}
\label{subsection: free algebras}

Let $B$ be an ordered blueprint with ambient semiring $B^+$, underlying monoid $B^\bullet$, and partial order $\leq$. The \emph{free algebra $B[T]$ over $B$} consists of the ambient semiring 
\[
 B[T]^+ \ = \ \big\{ \, \sum_{i=0}^n r_iT^i \, \big| \, r_i\in B^+ \, \big\},
\]
with respect to the usual addition and multiplication rules for polynomials, the underlying monoid
\[
 B[T]^\bullet \ = \ \{ \, aT^i \, | \, a\in B^\bullet \, \}
\]
of monomials in $B^+[T]$ with coefficients in $B^\bullet$, and the partial order generated by the relations
\[\textstyle
 rT^n \ \leq \ sT^n \qquad \text{whenever} \qquad r \ \leq \ s
\]
for $r,s\in B^+$. The universal property for $B[T]$ is as follows, cf.\ \cite[Lemma 5.5.2]{Lorscheid18}. 

\begin{lemma}
 For every morphism of ordered blueprints $f:B^+\to C^+$ and every element $a\in C$, there is a unique morphism of ordered blueprints $g:B[T]\to C$ such that $g(r)=f(r)$ for $r\in B^+$ and $g(T)=a$.
\end{lemma}

\begin{ex}
 A typical element of the free algebra $\T[T]$ over the tropical hyperfield is of the form $\sum r_iT^i$ where $r_i\in\N[\T^\times]$ is a formal sum $r_i=\sum a_k$ of tropical numbers $a_k\in\T^\times$. For example, we have
 \[
  T^2+T+1 \ \leq \ T^2+T+T+1 \ = \ (T+1)^2
 \]
 since $T\leq T+T$, but equality does not hold in $\T[T]^+$.

 A typical element of $\S[T]$ is of the form $\sum r_iT^i$ where $r_i\in\N[\{1,-1\}]$ is a formal sum of the form $r_i=1+\dotsb+1-1-\dotsb-1$. For example, we have
 \[
  T^2-1 \ \leq \ T^2+T-T-1 \ = \ (T+1)(T-1)
 \]
 since $0\leq T-T$, but equality does not hold in $\S[T]^+$.
\end{ex}

\subsection{Polynomial hyperrings, revisited}
\label{subsection: polynomial hyperrings revisited}

Let $F$ be a hyperfield and $B$ the associated ordered blueprint. Then every polynomial $\sum c_iT^i$ over $F$ is tautologically an element of the semiring $B[T]^+=\N[F^\times]$. This identifies $\Pol(F)$ with a subset of $B[T]^+$, which can be recovered from $B[T]$ as follows.

Let $B$ be an ordered blueprint. A \emph{polynomial over $B$} is an element of $B[T]^+$ of the form $p=\sum c_iT^i$ with $c_i\in B^\bullet$. We denote by $\Pol(B)$ the subset of polynomials in $B[T]^+$. 

If $B$ is the ordered blueprint associated with a hyperfield $F$, then $\Pol(F)=\Pol(B)$ as subsets of $B[T]^+$. Moreover, we obtain the following reinterpretation of the hyperaddition and hypermultiplication of polynomials over $F$:
\begin{align*}
 p_1\hyperplus p_2 \ &= \ \big\{ \, q\in \Pol(B) \, \big| \, q\leq p_1+p_2 \, \big\}, \\
 p_1\hyperdot p_2 \ &=  \ \big\{ \, q\in \Pol(B) \, \big| \, q\leq p_1\cdot p_2 \, \big\},
\end{align*}
where $p_1+p_2$ and $p_1\cdot p_2$ are, respectively, the sum and product of $p_1$ and $p_2$ as elements of $B[T]^+$. In other words, for $p_1,p_2,q\in\Pol(F)=\Pol(B)$ we have $q\in p_1\hyperplus p_2$ if and only if $q\leq p_1+p_2$ and $q\in p_1\hyperdot p_2$ if and only if $q\leq p_1\cdot p_2$.

\subsection{Roots of polynomials over ordered blueprints}
\label{subsection: roots of polynomials over ordered blueprints}

To close the circle of ideas, we reformulate the notions of roots and their multiplicities in our newly developed formalism and then extend these notions to a more general class of ordered blueprints than hyperfields. For this purpose, we introduce the notion of a pasture, which is an algebraic structure closely connected to the `foundation' of a matroid, cf.\ \cite{Baker-Lorscheid18}. There are several equivalent definitions of pastures. In this text, we realize them as a particular type of ordered blueprints.

We begin with some preliminary notions. We denote by $B^\times$ the group of multiplicatively invertible elements of $B$. An \emph{ordered blue field} is a nonzero ordered blueprint $B$ such that $B=B^\times\cup\{0\}$. An ordered blueprint $B$ is \emph{reversible} if it contains an element $\epsilon$ with $\epsilon^2=1$ such that every relation $a\leq b+r$ where $a,b\in B^\bullet$ and $r\in B^+$ implies $\epsilon b\leq \epsilon a+r$. As shown in \cite[Lemma 5.6.34]{Lorscheid18}, $\epsilon$ is uniquely determined by this property and for every element $a\in B^\bullet$ there is a unique element $b\in B^\bullet$ (namely $b=\epsilon a$) such that $0\leq a+b$. 

\begin{df}
 A \emph{pasture} is a reversible ordered blue field $B$ whose partial order is generated by relations of the form $c\leq a+b$ with $a,b,c\in B$ and such that the natural map $\N[B^\times]\to B^+$ is bijective.
\end{df}

Note that the ordered blueprint $B$ associated to a hyperfield $F$ is a pasture. Clearly, $B$ is an ordered blue field. The reversibility axiom \ref{HG3} for $F$ implies that $B$ is reversible. The last property follows from the fact that the partial order $\leq$ is generated by the relations $c\leq a+b$ for which $c\in a+b$ in $F$.

We extend the notions of roots and their multiplicities to polynomials from hyperfields to pastures.

\begin{df}
 Let $B$ be a pasture, let $a\in B^\bullet$, and let $p=\sum c_iT^i$ be a polynomial over $B$. Let $p(a)$ denote the element $\sum c_ia^i$ of $B^+$. Then $a$ is a \emph{root of $p$} if $0\leq p(a)$.
 
 If $a$ is not a root of $p$, we say that the \emph{multiplicity $\mult_a(p)$ of $a$} is $0$. If $a$ is a root of $p$, we define
 \[
  \mult_a(p) \ = \ 1+\max \big\{ \, \mult_a(q) \, \big| \, p\leq (T+\epsilon a)q \, \big\}.
 \]
\end{df}


Lemma \ref{lem: roots} generalizes to pastures $B$, with the same proof. Namely, $a\in B^\bullet$ is the root of a polynomial $p\in\Pol(B)$ if and only if there is a $q\in\Pol(B)$ such that $p\leq(T+\epsilon a)q$. 

Proposition~\ref{propA} also generalizes to pastures, with the same proof. Let $B$ be the ordered blueprint associated with a field $K$ (cf.\ section \ref{subsubsection: fields as ordered blueprints}) and $f:B\to C$ a morphism to a pasture $C$. Let $p=\sum c_iT^i\in\Pol(B)$ and denote by $\overline p=\sum f(c_i)T^i$ the image of $p$ in $\Pol(C)$.  Then for all $b\in C^\bullet$ we have
\[
 \mult_b(\overline p) \ \geq \ \sum_{\substack{a\in B^\bullet\text{ with}\\ f(a)=b}} \mult_a(p).
\]

\begin{small}
 \bibliographystyle{plain}
 \bibliography{hyperfield}
\end{small}

\end{document}